\documentclass{amsart}
\usepackage{amsthm,amsmath,amsfonts,amssymb,array,enumerate,verbatim,abstract}
\usepackage[all]{xy}
\usepackage{url}
\usepackage[small,nohug,heads=vee]{diagrams}
\diagramstyle[labelstyle=\scriptstyle]
\newcommand{\isomto}{\overset{\sim}{\rightarrow}}

\newcommand{\CC}{\mathbb{C}}
\newcommand{\ZZ}{\mathbb{Z}}

\newcommand{\QQ}{\mathbb{Q}}
\newcommand{\cO}{\mathcal{O}}

\newcommand{\fp}{\mathfrak{p}}
\newcommand{\fm}{\mathfrak{m}}
\newcommand{\fW}{\mathfrak{W}}
\newcommand{\cW}{\mathcal{W}}
\newcommand{\cK}{\mathcal{K}}
\newcommand{\cS}{\mathcal{S}}
\newcommand{\cA}{\mathcal{A}}
\newcommand{\cZ}{\mathcal{Z}}
\DeclareMathOperator{\Hom}{Hom}

\DeclareMathOperator{\Ind}{Ind}
\DeclareMathOperator{\cInd}{c-Ind}
\DeclareMathOperator{\End}{End}
\DeclareMathOperator{\Rep}{Rep}

\DeclareMathOperator{\Spec}{Spec}

\newtheorem{thm}{Theorem}[section]
\newtheorem{lemma}[thm]{Lemma}
\newtheorem{prop}[thm]{Proposition}
\newtheorem{cor}[thm]{Corollary}
\newtheorem{defn}[thm]{Definition}

\newtheorem{rmk}[thm]{Remark}
\newtheorem{obs}[thm]{Observation}

\title{Interpolating Local Constants in Families}
\author{Gilbert Moss}
\date{\today}
\begin{document}
\maketitle

\section{Introduction}
Let $G= GL_n(F)$, let $k$ be an algebraically closed field of characteristic $\ell$, and $W(k)$ its ring of Witt vectors. By an $\ell$-adic family of representations we mean an $A[G]$-module $V$ where $A$ is a commutative $W(k)$-algebra with unit; then each point $\fp$ of $A$ gives a $\kappa(\fp)[G]$-module $V\otimes_A\kappa(\fp)$ where $\kappa(\fp)$ denotes the residue field at $\fp$. In \cite{eh}, Emerton and Helm conjecture a local Langlands correspondence for $\ell$-adic families of admissible representations. To any continuous Galois representation $\rho:G_F\rightarrow GL_n(A)$, they conjecturally associate an admissible smooth $A[G]$-module $\pi(\rho)$, which interpolates the local Langlands correspondence for points $A\rightarrow \kappa$ with $\kappa$ characteristic zero. They prove that any $A[G]$-module which is subject to this interpolation property and a short list of representation-theoretic conditions (see \cite[Thm 6.2.1]{eh}) must be unique.

In \cite{h_whitt}, Helm further investigates the structure of $\pi(\rho)$ by taking the list of representation-theoretic conditions in \cite[Thm 6.2.1]{eh} as a starting point for the theory of ``co-Whittaker'' $A[G]$-modules (see Section \ref{cowhittakermodulesbernsteincenter} below for the definitions). Using this theory, he is able to reformulate the conjecture in terms of the existence of a certain homomorphism between the integral Bernstein center and a universal deformation ring (\cite[Thm 7.8]{h_whitt}).

Roughly speaking, representations of $GL_n(F)$ over $\CC$ are completely determined by data involving only local constants (\cite{hen_converse}), and in particular the bijections of the classical local Langlands correspondence are uniquely determined using $L$- and epsilon-factors (see, for example, \cite{djiang}). However, $L$- and epsilon-factors are absent from the the local Langlands correspondence in families. Thus it is natural to ask whether it is possible to attach $L$- and epsilon-factors to an $\ell$-adic family such as $\pi(\rho)$ as in \cite{eh}, or more generally any co-Whittaker $A[G]$-module, in a way that interpolates the $L$- and epsilon factors at each point. 

Over $\CC$, $L$-factors $L(\pi,X)$ arise as the greatest common denominator of the zeta integrals $Z(W,X;j)$ of a representation $\pi$ as $W$ varies over the space $\cW(\pi,\psi)$ of Whittaker functions (see Sections \ref{whittakerandkirillovfunctions}, \ref{definitionofthezetaintegrals} for definitions). Epsilon-factors $\epsilon(\pi,X,\psi)$ are the constant of proportionality (i.e. not depending on $W$) in a functional equation relating the modified zeta integral $\frac{Z(W,X)}{L(\pi,X)}$ to its pre-composition with a Fourier transform. Here, the formal variable $X$ replaces the complex variable $q^{-(s+\frac{n-1}{2})}$ appearing in \cite{jps1} and other literature, and we consider these objects as formal series.

It appears difficult to construct $L$-factors in a way compatible with arbitrary change of coefficients. To see this, consider the following simple example: let $q\equiv 1\mod \ell$, and let $\chi_1$, $\chi_2:F^{\times}\rightarrow W(k)^{\times}$ be smooth characters such that $\chi_1$ is unramified but $\chi_2$ is ramified, and such that $\chi_1 \equiv \chi_2\mod \ell$. Following the classical procedure (see for example \cite[23.2]{bh}) for finding a generator of the fractional ideal of zeta integrals, we get $L(\chi_i,X)\in W(k)(X)$ and find that $L(\chi_1,X)=\frac{1}{1-\chi_1(\varpi_F)X}$, and $L(\chi_2,X)=1$. Now let $A$ be the Noetherian local ring $\{(a,b)\in W(k)\times W(k): a\equiv b \mod \ell\}$, which has two characteristic zero points $\fp_1$, $\fp_2$ and a maximal ideal $\ell A$. Let $\pi$ be the $A[F^{\times}]$-module $A$, with the action of $F^{\times}$ given by $x\cdot (a,b) = (\chi_1(x)a,\chi_2(x)b)$. Interpolating $L(\chi_1,X)$ and $L(\chi_2,X)$ would mean finding an element $L(\pi,X)$ in $A[[X]][X^{-1}]$ such that $L(\pi,X)\equiv L(\chi_i,X) \mod \ell$ for $i=1,2$, but such a task is impossible because $L(\chi_1,X)$ and $L(\chi_2,X)$ are different mod $\ell$.

On the other hand, zeta integrals themselves seem to be much more well-behaved with respect to specialization. Classically, zeta integrals form elements of the quotient field $\CC(X)$ of $\CC[X,X^{-1}]$. Our first result is identifying, for more arbitrary coefficient rings $A$, the correct fraction ring in which our naive generalization of zeta factors will live:
\begin{thm}
\label{introrationality}
Suppose $A$ is a Noetherian $W(k)$-algebra. Let $S$ be the multiplicative subset of $A[X,X^{-1}]$ consisting of polynomials whose first and last coefficients are units. Then if $V$ is a co-Whittaker $A[G]$-module, $Z(W,X;j)$ lies in the fraction ring $S^{-1}(A[X,X^{-1}])$ for all $W\in \cW(V,\psi)$ and for $0\leq j \leq n-2$.
\end{thm}
The proof of rationality in the setting of representations over a field relies on a useful decomposition of a Whittaker function into ``finite'' functions (\cite[Prop 2.2]{jps1}). In the setting of rings, such a structure theorem is lacking, but certain elements of its proof can be translated into a question about the finiteness of the $(n-1)$st Bernstein-Zelevinsky derivative. This finiteness property, combined with a simple translation property of the zeta integrals, yields Theorem \ref{introrationality} (see \S \ref{proofofrationality}).

Classically, zeta integrals satisfy a functional equation which does not involve dividing by the $L$-factor. The constant of proportionality in this functional equation is called the gamma-factor and equals $\epsilon(\pi,X,\psi)\frac{L(\pi^{\iota},\frac{1}{q^nX})}{L(\pi,X)}$, when the $L$-factor makes sense. Our second main result is that gamma-factors interpolate in $\ell$-adic families (see \S \ref{analogueoffouriertransform} for details on the notation):
\begin{thm}
\label{mainthm}
Suppose $A$ is a Noetherian $W(k)$-algebra and suppose $V$ is a primitive co-Whittaker $A[G]$-module. Then there exists a unique element $\gamma(V,X,\psi)$ of $S^{-1}(A[X,X^{-1}])$ such that
$$Z(W,X;j)\gamma(V,X,\psi) = Z(\widetilde{w'W},\frac{1}{q^nX};n-2-j)$$ for any $W\in \cW(V,\psi)$ and for any $0\leq j \leq n-2$.
\end{thm}
To prove Theorem \ref{mainthm} we use the theory of the integral Bernstein center to reduce to the characteristic zero case of \cite{jps1}. 

The question of interpolating local constants in $\ell$-adic families has been investigated in a simple case by Vigneras in \cite{vig_epsilon}. For supercuspidal representations of $GL_2(F)$ over $\overline{\QQ_{\ell}}$, Vigneras notes in \cite{vig_epsilon} that it is known that epsilon factors define elements of $\overline{\ZZ_{\ell}}$, and proves that for two supercuspidal integral representations to be congruent modulo $\ell$ it is necessary and sufficient that they have epsilon factors which are congruent modulo $\ell$ (we call a representation with coefficients in a local field $E$ integral if it stabilizes an $\cO_E$-lattice). The classical epsilon and gamma factors are equal in the supercuspidal case, so when the specialization of an $\ell$-adic family at a characteristic zero point is supercuspidal, the gamma factor we construct in this paper specializes to the epsilon factor of \cite{jps1, vig_epsilon}.  Since two representations $V_1$, $V_2$ over $\cO_E$ which are congruent mod $\fm_E$ define a family $V_1 \times_{\overline{V}} V_2$ over the connected $W(k)$-algebra $\cO_E\times_{k_E} \cO_E$, Theorems \ref{introrationality} and \ref{mainthm} give the following corollary (implying the ``necessary'' part of \cite{vig_epsilon}):
\begin{cor}
Let $K$ denote the fraction field of $W(k)$. If $\pi$ and $\pi'$ are absolutely irreducible integral representations of $GL_n(F)$ over a coefficient field $E$ which is a finite extension of $K$, then:
\begin{enumerate}
\item $\gamma(\pi, X, \psi)$ and $\gamma(\pi', X, \psi)$ have coefficients in the fraction ring \\ $S^{-1}(\cO_E[X,X^{-1}])$.
\item If $\fm_E$ is the maximal ideal of $\cO_E$, and $\pi \equiv \pi' \mod\fm_E$, then $\gamma(\pi,X,\psi)\equiv \gamma(\pi',X,\psi)\mod\fm_E$.
\end{enumerate}
\end{cor}

The question of extending the theory of zeta integrals to the $\ell$-modular setting has been investigated in \cite{minguez}, and very recently in \cite{km} for the Rankin-Selberg integrals. The question of deforming local constants over polynomial rings over $\CC$ has been investigated by Cogdell and Piatetski-Shapiro in \cite{cogshap}, and the techniques of this paper owe much to those in \cite{cogshap}.

Analogous to the results of Bernstein and Deligne in \cite{bd} for $\Rep_{\CC}(G)$, Helm shows in \cite[Thm 10.8]{h_bern} that the category $\Rep_{W(k)}(G)$ has a decomposition into full subcategories known as blocks.  Our third main result is constructing for each block a gamma factor which is universal in the sense that it gives rise via specialization to the gamma factor for any co-Whittaker module in that block. We will now state this result more precisely.

Each block of the category $\Rep_{W(k)}(G)$ corresponds to a primitive idempotent in the Bernstein center $\mathcal{Z}$, which is defined as the ring of endomorphisms of the identity functor.  It is a commutative ring whose elements consist of collections of compatible endomorphisms of every object, each such endomorphism commuting with all morphisms.  Choosing a primitive idempotent $e$ of $\mathcal{Z}$, the ring $e\mathcal{Z}$ is the center of the subcategory $e\cdot \Rep_{W(k)}(G)$ of representations satisfying $eV=V$. The ring $e\cZ$ has an interpretation as the ring of regular functions on an affine algebraic variety over $W(k)$, whose $k$-points are in bijection with the set of unramified twists of a fixed conjugacy class of supercuspidal supports in $\Rep_k(G)$. See \cite{h_bern} for details. In \cite{h_whitt}, Helm determines a ``universal co-Whittaker module'' with coefficients in $e\cZ$, denoted here by $e\fW$, which gives rise to any co-Whittaker module via specialization (see Proposition \ref{dominance} below). By applying our theory of zeta integrals to $e\fW$ we get a gamma factor which is universal in the same sense:
\begin{thm}
\label{introunivgamma}
Suppose $A$ is any Noetherian $W(k)$-algebra, and suppose $V$ is a primitive co-Whittaker $A[G]$-module. Then there is a primitive idempotent $e$, a homomorphism $f_V:e\cZ\rightarrow A$, and an element $\Gamma(e\fW,X,\psi)\in S^{-1}(e\cZ[X,X^{-1}])$ such that $\gamma(V,X,\psi)=f_V(\Gamma(e\fW,X,\psi))$.
\end{thm}

Interpolating gamma factors of pairs may be the next step in obtaining a local converse theorem for $\ell$-adic families. By capturing the interpolation property, families of gamma factors might give an alternative characterization of the co-Whittaker module $\pi(\rho)$ appearing in the local Langlands correspondence in families.

The author would like to thank his advisor David Helm for suggesting this problem and for his invaluable guidance, Keenan Kidwell for his helpful conversations, and Peter Scholze for his helpful questions and comments at the MSRI summer school on New Geometric Techniques in Number Theory in 2013. He would also like to thank the referee for her/his very helpful comments and suggestions.

\subsection{Notation and Conventions}
\label{notationconventions}
Let $F$ be a finite extension of $\QQ_p$, let $q$ be the order of its residue field, and let $k$ be an algebraically closed field of characteristic $\ell$, where $\ell\neq p$ is an odd prime.  Denote by $W(k)$ the ring of Witt vectors over $k$. The assumption that $\ell$ is odd is made so that $W(k)$ contains a square root of $q$. When $\ell=2$ all the arguments presented will remain valid, after possibly adjoining a square root of $q$ to $W(k)$. The letter $G$ or $G_n$ will always denote the group $GL_n(F)$. Throughout the paper $A$ will be a Noetherian commutative ring which is a $W(k)$-algebra, with additional properties in various sections, and $\kappa(\fp)$ will denote the residue field of a prime ideal $\fp$ of $A$.  For any locally profinite group $H$, $\Rep_A(H)$ denotes the category of smooth representations of $H$ over the ring $A$, i.e. $A[H]$-modules for which every element is stabilized by an open subgroup of $H$.  Even when this category is not mentioned, all representations are presumed to be smooth. When $H$ is a closed subgroup of $G$, we define the non-normalized induction functor $\Ind_H^G$ (resp. $\cInd_H^G$) $:\Rep_A(H)\rightarrow \Rep_A(G)$ sending $\tau$ to the smooth part of the $A[G]$-module, under right translation, of functions (resp. functions compactly supported modulo $H$) $f:G\rightarrow \tau$ such that $f(hg) = \tau(h)f(g)$, $h\in H$, $g\in G$.

The integral Bernstein center of \cite{h_bern} (see the discussion preceding Theorem \ref{introunivgamma}) will always be denoted by $\cZ$. If $V$ is in $\Rep_A(G)$, then it is also in $\Rep_{W(k)}(G)$, and we frequently use the Bernstein decomposition of $\Rep_{W(k)}(G)$ to interpret properties of $V$.

If $A$ has a nontrivial ideal $I$, then $I\cdot V$ is an $A[H]$-submodule of $V$, which shows that most content would be missing if we developed the representation theory of $\Rep_A(H)$ around the notion of irreducible objects, or simple $A[H]$-modules. Thus conditions appear throughout the paper which in the traditional setting are implied by irreducibility:
\begin{defn}
\label{specialterms}
$V$ in $\Rep_A(H)$ will be called 
\begin{enumerate}
\item \emph{Schur} if the natural map $A\rightarrow \End_{A[G]}(V)$ is an isomorphism;
\item \emph{$G$-finite} if it is finitely generated as an $A[G]$-module.
\item \emph{primitive} if there exists a primitive idempotent $e$ in the Bernstein center $\cZ$ such that $eV = V$.
\end{enumerate}
\end{defn}

We say a ring is connected if it has connected spectrum or, equivalently, no nontrivial idempotents, for example any local ring or integral domain. Note that if $A$ is connected, Corollary \ref{dominanceobservation} implies all co-Whittaker $A[G]$-modules are primitive.

Denote by $N_n$ the subgroup of $G_n$ consisting of all unipotent upper-triangular matrices. Let $\psi: F\rightarrow W(k)^{\times}$ be an additive character of $F$ with $\ker \psi = \mathfrak{p}$.  Then  $\psi$ defines a character on any subgroup of $N_n(F)$ by $$(u)_{i,j}\mapsto \psi(u_{1,2}+\dots+u_{n-1,n});$$ we abusively denote this character by $\psi$ as well.  

If $H$ is a subgroup normalized by another subgroup group $K$, and $\theta$ is a character of the group $H$, denote by $\theta^k$ the character given by $\theta^k(h) = \theta(khk^{-1})$ for $h\in H$, $k\in K$. For $V$ in $\Rep_A(H)$, denote by $V_{H,\theta}$ the quotient $V/V(H,\theta)$ where $V(H,\theta)$ is the sub-$A$-module generated by elements of the form $hv-\theta(h)v$ for $h\in H$ and $v\in V$; it is $K$-stable if $\theta^k=\theta$, $k\in K$. Given a standard Levi subgroup $M\subset G_n$ with unipotent radical $U$, and $\textbf{1}$ the trivial character, we denote by $J_M$ the non-normalized Jacquet functor $\Rep_A(G)\rightarrow \Rep_A(M):V\mapsto V_{U,\textbf{1}}$.

For each $m\leq n$, let $G_m$ denote $GL_m(F)$ and embed it in $G$ via $(\begin{smallmatrix}G_m & 0 \\ 0 & I_{n-m}\end{smallmatrix})$.  We let $\{1\} = P_1 \subset \cdots \subset P_n$ denote the mirabolic subgroups of $G_1\subset \cdots \subset G_n$, where $P_m$ is given by $\left\{(\begin{smallmatrix} g_{m-1} & x \\ 0 & 1\end{smallmatrix}):g_{m-1}\in G_{m-1},\ x\in F^{m-1}\right\}$.  We also have the unipotent upper triangular subgroup $U_m$ of $P_m$ given by $\left\{(\begin{smallmatrix} I_{m-1} & x \\ 0 & 1\end{smallmatrix}):x\in F^{m-1}\right\}$. In particular, $U_m \simeq F^{m-1}$ and $P_m = U_mG_{m-1}$. Note that this is different from the groups $N(r)$ defined in Proposition \ref{explicitderivative}.

Since $G_n$ contains a compact open subgroup whose pro-order is invertible in $W(k)$, there exists a unique (for that choice of subgroup) normalized Haar measure, defining integration on the space $ C_c^{\infty}(G,A)$ of smooth compactly supported functions $G\rightarrow A$ (\cite[I.2.3]{vig}).

\section{Representation Theoretic Background}
\label{representationtheory}

\subsection{Co-invariants and Derivatives}
\label{coinvariantsandderivatives}
As in \cite{eh,b-zI}, we define the following functors with respect to the character $\psi$.
\begin{align*}
\Phi^{+}:&\Rep_A(P_{n-1})\rightarrow \Rep_A(P_n) &
\Psi^{+}:&\Rep(G_{n-1})\rightarrow \Rep(P_n)\\
&V \mapsto \cInd_{P_{n-1}U_n}^{P_n}V \text{  ($U_n$ acts via $\psi$)} &&
V\mapsto V \text{  ($U_n$ acts trivially)} \\
\hat{\Phi}^{+}:&\Rep(P_{n-1})\rightarrow \Rep(P_n)&
\Psi^{-}:&\Rep(P_n)\rightarrow \Rep(G_{n-1})
\\
&V \mapsto \Ind_{P_{n-1}U_n}^{P_n}V&&
V\mapsto V/V(U_n,\textbf{1})
\\
\Phi^{-}:&\Rep(P_n)\rightarrow \Rep(P_{n-1})
\\
&V\mapsto V/V(U_n,\psi)
\end{align*}
Note that we give these functors the same names as the ones originally defined in \cite{b-z}, but we use the non-normalized induction functors, as in \cite{b-zI,eh}, because they are simpler for our purposes. As observed in \cite{eh}, these functors retain the basic adjointness properties proved in \cite[\S 3.2]{b-zI}.  This is because the methods of proof in \cite{b-z,b-zI} use properties of $l$-sheaves which carry over to the setting of smooth $A[G]$-modules where $A$ is a Noetherian $W(k)$-algebra.
\begin{prop}[\cite{eh},3.1.3]
\label{functorprop}
\begin{enumerate}[(1)]
\item The functors $\Psi^-$, $\Psi^+$, $\Phi^-$, $\Phi^+$, $\hat{\Phi}^+$ are exact.
\item $\Phi^+$ is left adjoint to $\Phi^-$, $\Psi^-$ is left adjoint to $\Psi^+$, and $\Phi^-$ is left adjoint to $\hat{\Phi}^+$.
\item $\Psi^-\Phi^+ = \Phi^-\Psi^+ = 0$
\item $\Psi^-\Psi^+$, $\Phi^-\hat{\Phi}^+$, and $\Phi^-\Phi^+$ are naturally isomorphic to the identity functor.
\item For each $V$ in $\Rep(P_n)$ we have an exact sequence $$0\rightarrow \Phi^+\Phi^-(V)\rightarrow V \rightarrow \Psi^+\Psi^-(V) \rightarrow 0.$$
\item{(Commutativity with Tensor Product)} If $M$ is an $A$-module and $F$ is $\Psi^-$, $\Psi^+$, $\Phi^-$, $\Phi^+$, or $\hat{\Phi}^+$, we have $F(V\otimes_A M)\cong F(V)\otimes_A M$.
\end{enumerate}
\end{prop}

For $1\leq m \leq n$ we define the $m$th derivative functor $$(-)^{(m)}:= \Psi^-(\Phi^-)^{m-1}:\Rep(P_n)\rightarrow \Rep(G_{n-m}).$$  This gives a functor $\Rep(G_n)\rightarrow \Rep(G_{n-m})$ by first restricting representations to $P_n$ and then applying $(-)^{(m)}$; this functor is also denoted $(-)^{(m)}$.  The zero'th derivative functor $(-)^{(0)}$ is the identity. We can describe the derivative functor $(-)^{(m)}$ more explicitly by using the following lemma on the transitivity of coinvariants:
\begin{lemma}[\cite{b-z} \S 2.32]
\label{transitivityofcoinvariants}
Let $H$ be a locally profinite group, $\theta$ a character of $H$, and $V$ a representation of $H$.  Suppose $H_1$, $H_2$ are subgroups of $H$ such that $H_1H_2=H$ and $H_1$ normalizes $H_2$.  Then
$$\left( V_{H_2,\theta|_{H_2}}\right)_{H_1,\theta|_{H_1}} = V_{H,\theta}.$$
\end{lemma}
Define $N(r)$ to be the group of matrices whose first $r$ columns are those of the identity matrix, and whose last $n-r$ columns are those of elements of $N_n$ (recall that $N_n$ is the group of unipotent upper triangular matrices).  For $2\leq r \leq n$ we have $U_rN(r) = N(r-1)$ and $U_r$ normalizes $N(r)$. As $N(r)$ is contained in $N_n$, we define $\psi$ on $N(r)$ via its superdiagonal entries. We can also define a character $\widetilde{\psi}$ on $N(r)$ slightly differently from the usual definition: $\widetilde{\psi}$ will be given as usual via $\psi$ on the last $n-r-1$ superdiagonal entries, but trivially on the $(r,r+1)$ entry, i.e. $$\widetilde{\psi}(x) := \psi(0 + x_{r+1,r+2} +\dots + x_{n-1,n})\text{ for }x\in N(r).$$

The functors $(\Phi^-)^m$ and $(-)^{(m)}$, defined above, can be described more explicitly. Let $m=n-r$. By applying Lemma \ref{transitivityofcoinvariants} repeatedly with $H_1 = U_r$, and $H_2 = N(r-1)$, we get
\begin{prop}[\cite{vig} III.1.8]
\label{explicitderivative}
\begin{enumerate}[(1)]
\item   $(\Phi^-)^mV$ equals the module of coinvariants $V/V(N(n-m),\psi)$.
\item   $V^{(m)}$ equals the module of coinvariants $V/V(N(n-m),\widetilde{\psi})$.
\end{enumerate}
\end{prop}
In particular, if $m=n$, this gives $V^{(n)} = V/V(N_n,\psi)$.  Note that $V^{(n)}$ is simply an $A$-module.

\subsection{Whittaker and Kirillov Functions}
\label{whittakerandkirillovfunctions}
The character $\psi:N_n\rightarrow A^{\times}$ defines a representation of $N_n$ in the $A$-module $A$, which we also denote by $\psi$.  By Proposition \ref{explicitderivative} we have $\Hom_A(V^{(n)},A) = \Hom_{N_n}(V,\psi)$.
\begin{defn}
For $V$ in $\Rep_A(G_n)$, we say that $V$ is of Whittaker type if $V^{(n)}$ is free of rank one as an $A$-module. As in \cite[Def 3.1.8]{eh}, if $A$ is a field we refer to representations of Whittaker type as generic.
\end{defn}
If $V$ is of Whittaker type, $\Hom_{N_n}(V,\psi)$ is free of rank one, so we may choose a generator $\lambda$. The image of $\lambda$ under the Frobenius reciprocity isomorphism $\Hom_{N_n}(V,\psi)\isomto \Hom_{G_n}(V,\Ind_{N_n}^{G_n}\psi)$ is the map $v\mapsto W_v$ where $W_v(g) = \lambda(gv)$. The $A[G]$-module formed by the image of the map $v\mapsto W_v$ is independent of the choice of $\lambda$. 
\begin{defn}
The image of the homomorphism $V\rightarrow \Ind_{N_n}^{G_n}\psi$ is called the space of Whittaker functions of $V$ and is denoted $\cW(V,\psi)$ or just $\cW$.
\end{defn}
Choosing a generator of $V^{(n)}$ and allowing $N_n$ to act via $\psi$, we get an isomorphism $V^{(n)}\isomto \psi$.  Composing this with the natural quotient map $V\rightarrow V^{(n)}$ gives an $N_n$-equivariant map $V\rightarrow \psi$, which is a generator $\lambda$. Note that the map $V\rightarrow \cW(V,\psi)$ is surjective but not necessarily an isomorphism, unlike the setting of \emph{irreducible} generic representations with coefficients in a field. Different $A[G]$-modules of Whittaker type can have the same space of Whittaker functions:

\begin{lemma}
\label{whittakerunderhomomorphism}
Suppose $V'$, $V$ in $\Rep_A(G)$ are of Whittaker type, and suppose there is a $G$-equivariant homomorphism $\alpha:V'\rightarrow V$ such that $\alpha^{(n)}:(V')^{(n)}\rightarrow V^{(n)}$ is an isomorphism.  Then 
$\cW(V',\psi)$ is the subrepresentation of $\cW(V,\psi)$ given by $\cW(\alpha(V'),\psi)$.
\end{lemma}
\begin{proof}  Let $q':V'\rightarrow V'/V'(N_n,\psi)$ and $q: V\rightarrow V/V(N_n,\psi)$ be the quotient maps.  Choosing a generator for $V^{(n)}$ gives isomorphisms $\eta$, $\eta'$ such that the following diagram commutes.
\begin{diagram}
 V                            &               \rTo^{q}                        &V^{(n)}  &\rTo^{\eta} &A\\
  \dTo^{\alpha}                &                                                               &\dTo^{\alpha^{(n)}}    &\ruTo_{\eta'}     \\
V'                 &             \rTo^{q'}   &(V')^{(n)}
\end{diagram} Given $v'\in V'$ we get $$W_{\alpha(v')}(g) = \eta(q(g\alpha v')) = \eta((q\circ \alpha) (gv'))=\eta'(q'(gv')) = W_{v'}(g),\ \ g\in G.$$ This shows $\cW(V',\psi) = \cW(\alpha(V'),\psi)\subset \cW(V,\psi)$.
\end{proof}

If $V$ in $\Rep_A(G_n)$ is Whittaker type and $v\in V$, we will denote by $W_v|_{P_n}$ the restriction of the function $W_v$ to the subgroup $P_n\subset G_n$.
\begin{defn}
The image of the $P_n$-equivariant homomorphism $V\rightarrow \Ind_{N_n}^{P_n}\psi:v\mapsto W_v|_{P_n}$ is called the Kirillov functions of $V$ and is denoted $\cK(V,\psi)$ or just $\cK$.
\end{defn}

The following properties of the Kirillov functions are well known for $\Rep_{\CC}(G)$, but we will need them for $\Rep_A(G)$:
\begin{prop}
\label{schwarzandkirillov}
Let $V$ be of Whittaker type in $\Rep_A(P_n)$, and choose a generator of $V^{(n)}$ in order to identify $V^{(n)}$ with $A$.  Then the following hold:
\begin{enumerate}[(1)]
\item $(\Phi^+)^{n-1}V^{(n)} = \cInd_{N_n}^{P_n}\psi$ and $(\hat{\Phi}^+)^{n-1}V^{(n)} = \Ind_{N_n}^{P_n}\psi$.
\item The composition $(\Phi^+)^{n-1}V^{(n)}\rightarrow V \rightarrow \Ind_{N_n}^{P_n}\psi$ differs from the inclusion $\cInd_{N_n}^{P_n}\psi\hookrightarrow \Ind_{N_n}^{P_n}\psi$ by multiplication with an element of $A^{\times}$.
\item The Kirillov functions $\cK(V,\psi)$ contains $\cInd_{N_n}^{P_n}\psi$ as a sub-$A[P_n]$-module.
\end{enumerate}
\end{prop}
\begin{proof} The proof in \cite{b-z} Proposition 5.12 (g) works to prove (1) in this context.

Let $\mathfrak{S}=(\Phi^+)^{n-1}V^{(n)}$. There is an embedding $\mathfrak{S}\rightarrow V$ by Proposition \ref{functorprop} (5); denote by $t$ the composition $\mathfrak{S}\rightarrow V \rightarrow \Ind\psi$. Then $t^{(n)}:\mathfrak{S}^{(n)}\rightarrow \Ind\psi^{(n)}$ is a nonzero homomorphism between free rank one $A$-modules, hence given by multiplication with an element $a$ of $A$. By Proposition \ref{functorprop} (6), For any maximal ideal $\fm$ of $A$, $t^{(n)}\otimes \kappa(\fm)$ must be an isomorphism because it is a nonzero element of $$\Hom_{\kappa(\fm)}((S(V)\otimes\kappa(\fm))^{(n)},(\Ind\psi\otimes\kappa(\fm))^{(n)}) = \kappa(\fm).$$ Thus $a$ is nonzero in $\kappa(\fm)$ for all $\fm$, hence a unit, so $t^{(n)}$ is an isomorphism. On the other hand there is the natural embedding $\cInd\psi\rightarrow \Ind\psi$, which we will denote $s$. Since $s^{(n)}$ is an isomorphism by \cite[Prop 3.2 (f)]{b-zI}, we have $s^{(n)} = ut^{(n)}$ for some $u\in A^{\times}$. Thus, if $K:= \ker(s-ut)$ then $K^{(n)}=S(V)^{(n)}=V^{(n)}$, whence $\Hom_P(S(V)/K,\Ind\psi) \cong \Hom_A((S(V)/K)^{(n)},A) = \Hom_A(\{0\},A) = 0$, which implies $s-ut \equiv 0$.

To prove (3), note that since $\cK(V,\psi)$ is defined to be the image of the map $V\rightarrow \Ind_{N_n}^{P_n}\psi$, this follows from $(2)$.
\end{proof}

\begin{defn}[\cite{eh},\S 3.1 ]
If $V$ is in $\Rep(P_n)$, the image of the natural inclusion $(\Phi^+)^{n-1}V^{(n)}\rightarrow V$ is called the Schwartz functions of $V$ and is denoted $\cS(V)$.  For $V$ in $\Rep(G_n)$ we also denote by $\cS(V)$ the Schwartz functions of $V$ restricted to $P_n$.
\end{defn}

We can ask how the functor $\Phi^-$ is reflected in the Kirillov space of a representation.  First we observe that $\Phi^-$ commutes with the functor $\cK$:

\begin{lemma}
\label{phicommuteswithK}
For $0\leq m\leq n$, the $A[P_m]$-modules $\cK((\Phi^-)^{n-m}V,\psi)$ and $(\Phi^-)^{n-m}\cK(V,\psi)$ are identical.
\end{lemma} 
\begin{proof}
The image of the $P_{n-m}$-submodule $V(N(m),\psi)$ in the map $V\rightarrow \cK$ equals the submodule $\cK(N(m),\psi)$. The lemma then follows from Proposition \ref{explicitderivative}
\end{proof}

Following \cite{cogshap}, we can explicitly describe the effect of the functor $\Phi^-$ on the Kirillov functions $\cK$.  Recall that $\cK(U_n,\psi)$ denotes the $A$-submodule generated by $\{uW-\psi(u)W:u\in U_n,\ W\in \cK\}$ and $\Phi^-\cK := \cK/\cK(U_n,\psi)$.

\begin{prop}[\cite{cogshap} Prop 1.1]
\label{zerorestricted}
$$\cK(U_n,\psi) = \{W\in \cK : W \equiv 0 \text{ on the subgroup } P_{n-1}\subset P_n\}.$$
\end{prop}
\begin{proof}
The proof of \cite[Prop 1.1]{cogshap} carries over in this setting. It utilizes the Jacquet-Langlands criterion for an element $v$ of a representation $V$ to be in the subspace $V(U_{n_i},\psi)$, which remains valid over more general coefficient rings $A$ because all integrals are finite sums.
\end{proof}

Thus $\Phi^-$ has the same effect as restriction of functions to the subgroup $P_{n-1}$ inside $P_n$:
$$\Phi^-\cK \cong \left \{W\left(\begin{smallmatrix}   p & 0 \\ 0 & 1  \end{smallmatrix}\right) : W \in \cW(V,\psi),\ p\in P_{n-1}\right\}.$$
By applying for each $m=1,\dots,n-2$ the argument of \cite[Prop 1.1]{cogshap} to the $P_{n-m+1}$ representation
$$ \left \{W\left(\begin{smallmatrix}   p & 0 \\ 0 & I_{m-1}  \end{smallmatrix}\right) : W \in \cW(V,\psi),\ p\in P_{n-m+1}\right\}$$ instead of to $\cK$, we can describe $(\Phi^-)^m\cK$:

\begin{cor}
\label{explicitphi}
For $m = 1,\dots, n-1$,
$$(\Phi^-)^m\cK \cong \left \{W\left(\begin{smallmatrix}   p & 0 \\ 0 & I_{m}  \end{smallmatrix}\right) : W \in \cW(V,\psi),\ p\in P_{n-m}\right\}.$$
\end{cor}

\subsection{Partial Derivatives}
Given a product $H_1\times H_2$ of subgroups of $G$, and a character $\psi$ on the unipotent upper triangular elements of $H_2$, we can define ``partial'' versions of the functors $\Phi^{\pm}$, $\Psi^{\pm}$ as follows: given $V$ in $\Rep_A(H_1\times H_2)$, restrict it to a representation of $H_1 = \{1\}\times H_2$, then apply the functor $\Phi^{\pm}$ or $\Psi^{\pm}$, and observe that $H_1\times \{1\}$ acts naturally on the result, since it commutes with $\{1\}\times H_2$. More precisely:
\begin{align*}
\Phi^{+,2}:&\Rep_A(G_{n-m}\times P_{m-1})\rightarrow \Rep_A(G_{n-m}\times P_m)\\
&V \mapsto \cInd_{G_{n-m}\times P_{m-1}U_m}^{G_{n-m}\times P_m}(V) \text{, \ \    with $\{1\}\times U_m$ acting via $\psi$}\\
\hat{\Phi}^{+,2}:&\Rep(G_{n-m}\times P_{m-1})\rightarrow \Rep(G_{n-m}\times P_m)\\
&V \mapsto \cInd_{G_{n-m}\times P_{m-1}U_m}^{G_{n-m}\times P_m}(V)\\
\Phi^{-,2}:&\Rep(G_{n-m}\times P_m)\rightarrow \Rep(G_{n-m}\times P_{m-1})\\
&V\mapsto V/V(\{1\}\times U_m,\psi)\\
\Psi^{+,2}:&\Rep(G_{n-m}\times G_{m-1})\rightarrow \Rep(G_{n-m}\times P_m)\\
&V\mapsto V \text{  ($\{1\}\times U_m$ acts trivially)} \\
\Psi^{-,2}:&\Rep(G_{n-m}\times P_m)\rightarrow \Rep(G_{n-m}\times G_{m-1})\\
&V\mapsto V/V(\{1\}\times U_m,\textbf{1})
\end{align*}
Because $H_1\times \{1\}$ commutes with $\{1\}\times H_2$, we immediately get
\begin{lemma}
\label{partialfunctorprop}
The analogue of Proposition \ref{functorprop} (1)-(6) holds for $\Phi^{+,2}$, $\hat{\Phi}^{+,2}$, $\Phi^{-,2}$, $\Psi^{+,2}$, and $\Psi^{-,2}$.
\end{lemma}
\begin{defn}
We define the functor $(-)^{(0,m)}:\Rep_A(G_{n-m}\times G_m)\rightarrow \Rep_A(G_{n-m})$ to be the composition $\Psi^{-,2}(\Phi^{-,2})^{m-1}$.
\end{defn}
The proof of the following Proposition holds for $W(k)$-algebras $A$:
\begin{prop}[\cite{b-zII} Prop 6.7, \cite{vig} III.1.8]
\label{derivativefactorsthroughjacquet}
Let $M=G_{n-m}\times G_m$. For $0\leq m\leq n$ the $m$'th derivative functor $(-)^{(m)}$ is the composition of the Jacquet functor $J_M:\Rep(G_n)\rightarrow \Rep(G_{n-m}\times G_m)$ with the functor $(-)^{(0,m)}:\Rep_A(G_{n-m}\times G_m)\rightarrow \Rep_A(G_{n-m})$.
\end{prop}
\begin{lemma}
\label{toppartialderivativeembeds}
Let $V$ be in $\Rep_A(G_{n-m}\times G_m)$. Then $V$ contains an $A$-submodule isomorphic to $V^{(0,m)}$.
\end{lemma}
\begin{proof}
The image of the natural embedding $(\Phi^{+,2})^{m-1}\Psi^{+,2}(V^{(0,m)})\rightarrow V$, which is given by Proposition \ref{partialfunctorprop} (5), will be denoted $\cS^{0,2}(V)$. By Proposition \ref{partialfunctorprop} (4), the natural surjection $V\rightarrow V^{(0,m)}$ restricts to a surjection $\cS^{0,2}(V)\rightarrow V^{(0,m)}$. By Proposition \ref{partialfunctorprop} (6), the map of $A$-modules $\cS^{0,2}(V)\rightarrow V^{(0,m)}$ arises from the map $(\Phi^{+,2})^{m-1}\Psi^{+,2}(A)\rightarrow A$ by tensoring over $A$ with $V^{(0,m)}$. Take the $A$-submodule generated by any element of $(\Phi^{+,2})^{m-1}\Psi^{+,2}(A)$ that maps to the identity in $A$; then tensor with $V^{(0,m)}$.
\end{proof}

\subsection{Finiteness Results}
In this subsection we gather certain finiteness results involving derivatives, most of which are well-known when $A$ is a field of characteristic zero.

Let $H$ be any topological group containing a decreasing sequence $\{H_i\}_{i\geq 0}$ of open subgroups whose pro-order is invertible in $A$, and which forms a neighborhood base of the identity in $H$. If $V$ is a smooth $A[H]$-module we may define a projection $\pi_i:V\rightarrow V^{H_i}:v\mapsto \int_{H_i}hv$ for a Haar measure on $H_i$ where $H_i$ has total measure 1. The $A$-submodules $V_i:=\ker(\pi_i)\cap V^{H_{i+1}}$ then satisfy $\bigoplus_iV_i = V$.

\begin{lemma}[\cite{eh} Lemma 2.1.5, 2.1.6]
\label{admissibledirectsumfinitelygenerated}
A smooth $A[H]$-module $V$ is admissible if and only if each $A$-module $V_i$ is finitely generated. In particular, quotients of admissible $A[H]$-modules by $A[H]$-submodules are admissible.
\end{lemma}

Thus the following version of the Nakayama lemma applies to admissible $A[H]$-modules:
\begin{lemma}[\cite{eh} Lemma 3.1.9]
\label{nakayama}
Let $A$ be a Noetherian local ring with maximal ideal $\fm$, and suppose that $M$ is a submodule of a direct sum of finitely generated $A$-modules.  If $M/\fm M$ is finite dimensional then $M$ is finitely generated over $A$.
\end{lemma}
If $V$ is admissible, then it is $G$-finite if and only if $V/\fm V$ is $G$-finite. To see this, take $S\subset V/\fm V$ an $(A/\fm)[H]$-generating set, let $W$ be the $A[H]$-span of a lift to $V$. Since $V/W$ is admissible, we can apply Nakayama to each factor $(V/W)_i$ to conclude $V/W=0$. 

\begin{prop}[\cite{eh} 3.1.7]
\label{topderivativesmall}
Let $\kappa$ be a $W(k)$-algebra which is a field, and $V$ an absolutely irreducible admissible representation of $G_n$. Then $V^{(n)}$ is zero or one-dimensional over $\kappa$, and is one-dimensional if and only if $V$ is cuspidal.
\end{prop}

\begin{prop}[\cite{vig} II.5.10(b)]
\label{admissible+finite=finitelength}
Let $\kappa$ be a $W(k)$-algebra which is a field. If $V$ is a $\kappa[G]$-module, then $V$ is admissible and $G$-finite if and only if $V$ is finite length over $\kappa[G]$.
\end{prop}
\begin{proof}
Suppose $V$ is admissible and $G$-finite. If $\kappa$ were algebraically closed of characteristic zero (resp. characteristic $\ell$), this is \cite[4.1]{b-zI} (resp. \cite[II.5.10(b)]{vig}). Otherwise, let $\overline{\kappa}$ be an algebraic closure, then $V\otimes\overline{\kappa}$ is finite length, so $V$ is finite length.

If $V$ is finite length, so is $V\otimes_{\kappa}\overline{\kappa}$. Over an algebraically closed field of characteristic different from $p$, irreducible representations are admissible (\cite[3.25]{b-zI},\cite[II.2.8]{vig}). Since admissibility is preserved under taking extensions $V\otimes\overline{\kappa}$ being finite length implies it is admissible, hence $V$ is admissible. Thus we can reduce proving $G$-finiteness to proving that, given any exact sequence of admissible objects, $0\rightarrow W_0\rightarrow V \rightarrow W_1\rightarrow 0$ where $W_0$ and $W_1$ are $G$-finite, then $V$ is $G$-finite. But there is a compact open subgroup $U$ such that $W_0$ and $W_1$ are generated by $W_0^U$ and $W_1^U$, respectively. It follows that that $V$ is generated by $V^U$.
\end{proof}
\begin{lemma}
\label{derivativespreservefinitelength}
Let $\kappa$ be a $W(k)$-algebra which is a field. If $V$ is an absolutely irreducible $\kappa[G_n]$-module, then for $m\geq 0$, $V^{(m)}$ is finite length as a $\kappa[G_{n-m}]$-module.
\end{lemma}
\begin{proof}
We follow \cite[III.1.10]{vig}. Given $j$, $k$ positive integers, let $M=G_j\times G_k$ and let $P=MN$ be the associated standard parabolic subgroup. Given $\tau$ in $\Rep_{\kappa}(G_j)$ and $\sigma$ in $\Rep_{\kappa}(G_k)$, we define $\tau\times\sigma$ to be the normalized induction $\cInd_P(\delta_N^{1/2}(\sigma\otimes\tau))$ in $\Rep_{\kappa}(G_{j+k})$, where $\delta_N$ denotes the modulus character of $N$ (for the definition of $\delta_N$ see \cite[1.7]{b-zI}). There exists a multiset $\{\pi_1,\dots,\pi_r\}$ of irreducible cuspidals such that $V\subset \pi_1\times\cdots\times \pi_r$. The Liebniz formula for derivatives says that $(\pi_1\times\pi_2)^{(t)}$ has a filtration whose successive quotients are $\pi_1^{(t-i)}\times\pi_2^{(i)}$. Its proof, given in \cite[\S 7]{b-zI}, carries over in this generality. Then $V^{(m)}\subset (\pi_1\times\cdots\times \pi_r)^{(m)}$, which is finite length by induction, using Proposition \ref{topderivativesmall} combined with the Liebniz formula.
\end{proof}
\begin{prop}[\cite{h_whitt} Prop 9.15]
\label{jacquetpreservesadmissibility}
Let $M$ be a standard Levi subgroup of $G$. If $V$ in $\Rep_A(G)$ is admissible and primitive, then $J_MV$ in $\Rep_A(M)$ is admissible.
\end{prop}
\begin{cor}
\label{admissibilityofderivatives}
If $A$ is a local Noetherian $W(k)$-algebra and $V$ is admissible and $G$-finite, then $V^{(m)}$ is admissible and $G$-finite for $0\leq m \leq n$.
\end{cor}
\begin{proof}
Let $M = G_{n-m}\times G_m$. By Proposition \ref{derivativefactorsthroughjacquet}, $V^{(m)}=(J_MV)^{(0,m)}$, so by Lemma \ref{toppartialderivativeembeds}, there is an embedding $V^{(m)}\rightarrow J_MV$. Admissibility and $G$-finiteness mean $V$ is generated over $A[G]$ by vectors in $V^K$ for some compact open subgroup $K$. Since $V^K$ is finite over $A$, $eV^K$ is nonzero for only a finite set of primitive idempotents $e$ of the Bernstein center, so $eV\neq 0$ for at most finitely many primitive idempotents $e$ of the integral Bernstein center. Therefore, Proposition \ref{jacquetpreservesadmissibility} applies, showing $V^{(m)}$ embeds in an admissible module. Thus by Lemma \ref{nakayama}, we are reduced to proving the result for $\overline{V}:=V/\fm V$. Since $\overline{V}$ is admissible and $G$-finite, and $A/\fm$ is characteristic $\ell$, Lemma \ref{admissible+finite=finitelength} shows $\overline{V}$ is finite length, therefore it follows from Lemma \ref{derivativespreservefinitelength} that $\overline{V}^{(m)}$ is finite length. Applying Lemma \ref{admissible+finite=finitelength} once more, we have the result.
\end{proof}
Loosely speaking, the $(n-1)$st derivative describes the restriction of a $G_n$-representation to a $G_1$-representation (see Corollary \ref{explicitphi}). The next result shows that this restriction intertwines a finite set of characters:
\begin{thm}
\label{n-1derivfinite}
If $A$ is a local $W(k)$-algebra and $V$ in $\Rep_A(G)$ is admissible and $G$-finite, then $V^{(n-1)}$ is finitely generated as an $A$-module.
\end{thm}
\begin{proof} By Lemma \ref{nakayama} and Corollary \ref{admissibilityofderivatives} it is sufficient to show that $\overline{V}^{(n-1)}$ is finite over the residue field $\kappa$. We know $\overline{V}^{(n-1)}$ is $G$-finite and admissible by Corollary \ref{admissibilityofderivatives}, hence finite length as a $\kappa[G_1]$-module by Proposition \ref{admissible+finite=finitelength}. Since $G_1$ is abelian, all composition factors are $1$-dimensional, so $\overline{V}^{(n-1)}$ being finite length implies it is finite dimensional over $\kappa$.
\end{proof}

Since the hypotheses of being admissible and $G$-finite are preserved under localization by Proposition \ref{functorprop} (6), we can go beyond the local situation:
\begin{cor}
\label{localn-1derivativesfinite}
Let $A$ be a Noetherian $W(k)$-algebra and suppose that $V$ is admissible and $G$-finite. Then for every $\fp$ in $\Spec A$, $V_{\fp}^{(n-1)}$ is finitely generated as an $A_{\fp}$-module.
\end{cor}

\subsection{Co-Whittaker $A[G]$-Modules}
\label{cowhittakermodulesbernsteincenter}

In this subsection we define co-Whittaker representations and show that every admissible $A[G]$-module $V$ of Whittaker type contains a canonical co-Whittaker subrepresentation.
\begin{defn}[\cite{h_whitt} 3.3]
Let $\kappa$ be a field of characteristic different from $p$. An admissible smooth object $U$ in $\Rep_{\kappa}(G)$ is said to have essentially AIG dual if it is finite length as a $\kappa[G]$-module, its cosocle $\cos(U)$ is absolutely irreducible generic, and $\cos(U)^{(n)}=U^{(n)}$ (the cosocle of a module is its largest semisimple quotient).
\end{defn}

This condition is equivalent to $U^{(n)}$ being one-dimensional over $\kappa$ and having the property that $W^{(n)}\neq 0$ for any nonzero quotient $\kappa[G]$-module $W$ (see \cite[Lemma 6.3.5]{eh} for details).

\begin{defn}[\cite{h_whitt} 6.1]
An object $V$ in $\Rep_A(G)$ is said to be co-Whittaker if it is admissible, of Whittaker type, and $V\otimes_A \kappa(\fp)$ has essentially AIG dual for each $\fp$.
\end{defn}

\begin{prop}[\cite{h_whitt} Prop 6.2]
\label{cowhittimpliesschur}
Let $V$ be a co-Whittaker $A[G]$-module.  Then the natural map $A\rightarrow \End_{A[G]}(V)$ is an isomorphism.
\end{prop}

\begin{lemma}
\label{cowhittakercyclic}
Suppose $V$ is admissible of Whittaker type and, for all primes $\fp$, any non-generic quotient of $V\otimes\kappa(\fp)$ equals zero. Then $V$ is generated over $A[G]$ by a single element.
\end{lemma}
\begin{proof}
Let $x$ be a generator of $V^{(n)}$, and let $\tilde{x}\in V$ be a lift of $x$. If $V'$ is the $A[G]$-submodule of $V$ generated by $\tilde{x}$, then $(V/V')^{(n)}=0$. Since any non-generic quotient of $V\otimes \kappa(\fp)$ equals zero, $(V/V')\otimes \kappa(\fp) = 0$ for all $\fp$. Since $V/V'$ is admissible, we can apply Lemma \ref{nakayama} over the local rings $A_{\fp}$ to conclude $V/V'$ is finitely generated, then apply ordinary Nakayama to conclude it is zero.
\end{proof}
Thus every co-Whittaker module is admissible, Whittaker type, $G$-finite (in fact $G$-cyclic), and Schur, so satisfies the hypotheses of Theorem \ref{rationality}, below. Moreover, every admissible Whittaker type representation contains a canonical co-Whittaker submodule:
\begin{prop}
\label{cowhittakersub}
Let $V$ in $\Rep_A(G)$ be admissible of Whittaker type.  Then the sub-$A[G]$-module
$$T:=\ker(V \rightarrow \prod_{ \{U\subset V:\  (V/U)^{(n)}=0\}}V/U)$$
is co-Whittaker.
\end{prop}
\begin{proof}
$(V/T)^{(n)} = 0$ so $T$ is Whittaker type.  Since $V$ is admissible so is $T$. Let $\fp$ be a prime ideal and let $\overline{T}:=T\otimes\kappa(\fp)$. We show that $\cos(\overline{T})$ is absolutely irreducible and generic.  By its definition, $\cos(\overline{T}) = \bigoplus_j W_j$ with $W_j$ an irreducible $\kappa(\fp)[G]$-module.  Since the map $\overline{T}\rightarrow  \bigoplus_j W_j$ is a surjection and $(-)^{(n)}$ is exact and additive, the map $(\overline{T})^{(n)}\rightarrow \bigoplus_j W_j^{(n)}$ is also a surjection.  Hence $\dim_{\kappa(\fp)}(\bigoplus_j W_j^{(n)})\leq \dim_{\kappa(\fp)}(\overline{T}^{(n)})$.  Since $T$ is Whittaker type and $\overline{T}^{(n)} = \overline{T^{(n)}}$ is nonzero, there can only be one $j$ such that $W_j^{(n)}$ is potentially nonzero.  On the other hand, suppose some $W_j^{(n)}$ were zero, then $W_j$ is a quotient appearing in the target of the map $$\overline{V}\rightarrow \prod_{ \{U\subset \overline{V}:\  (\overline{V}/U)^{(n)}=0\}}\overline{V}/U,$$ hence as a quotient of $\overline{T}$ it would have to be zero, a contradiction.  Hence precisely one $W_j$ is nonzero. Now applying \cite[6.3.4]{eh} with $A$ being $\kappa(\fp)$ and $V$ being $\cos(\overline{T})$, we have that $\End_G(\cos(\overline{T}))\cong\kappa(\fp)$ hence absolutely irreducible.  It also shows that $\cos(\overline{T})^{(n)} = W_j^{(n)} \neq 0$.  Hence $\overline{T}^{(n)}=\cos(\overline{T})^{(n)}$. By Lemma \ref{cowhittakercyclic}, $\overline{T}$ is $\kappa(\fp)[G]$-cyclic; since it is admissible, it is finite length by Lemma \ref{admissible+finite=finitelength}.
\end{proof}

\subsection{The Integral Bernstein Center}
If $A$ is a Noetherian $W(k)$-algebra and $V$ is an $A[G]$-module, then in particular $V$ is a $W(k)[G]$-module, so we use the Bernstein decomposition of $\Rep_{W(k)}(G)$ to study $V$.

Let $\fW$ be the $W(k)[G]$-module $\cInd_{N_n}^{G_n}\psi$. If $e$ is a primitive idempotent of $\mathcal{Z}$, the representation $e\fW$ lies in the block $e\Rep_{W(k)}(G)$, and we may view it as an object in the category $\Rep_{e\cZ}(G)$. With respect to extending scalars from $e\cZ$ to $A$, the module $e\fW$ is ``universal'' in the following sense:
\begin{prop}[\cite{h_whitt} Thm 6.3]
\label{dominance}
Let $A$ be a Noetherian $e\cZ$-algebra.  Then $e\fW\otimes_{e\cZ} A$ is a co-Whittaker $A[G]$-module. Conversely, if $V$ is a primitive co-Whittaker $A[G]$ module in the block $e\Rep_{W(k)}(G)$, and $A$ is an $e\cZ$-algebra via $f_V:e\cZ \rightarrow A$, then there is a surjection $\alpha:\fW\otimes_{\cA,f_V}A\rightarrow V$ such that $\alpha^{(n)}: (\fW\otimes_{\cA,f_V}A)^{(n)}\rightarrow V^{(n)}$ is an isomorphism.
\end{prop}

If we assume $A$ has connected spectrum (i.e. no nontrivial idempotents), then the map $f_V:\cZ\rightarrow A$ would factor through a map $e\cZ\rightarrow A$ for some primitive idempotent $e$, hence:
\begin{cor}
\label{dominanceobservation}
If $A$ is a connected Noetherian $W(k)$-algebra and $V$ is co-Whittaker, then $V$ must be primitive for some primitive idempotent $e$.
\end{cor}
\begin{rmk}
\label{introstatementsremark}
Theorems \ref{introrationality}, \ref{mainthm}, and \ref{introunivgamma} remain true if the hypothesis that $V$ is primitive is replaced with the hypothesis that $A$ is connected.
\end{rmk}

\section{Zeta Integrals}
\label{zetaintegrals}
In this section we use the representation theory of Section \ref{representationtheory} to define zeta integrals and investigate their properties.

\subsection{Definition of the Zeta Integrals}
\label{definitionofthezetaintegrals}
We first propose a definition of the zeta integral which is the analog of that in \cite{jps1}, and then check that the definition makes sense.
\begin{defn}
\label{zetadefinition}
For $W\in \cW(V,\psi)$ and $0\leq j \leq n-2$, let $X$ be an indeterminate and define 
$$Z(W,X;j) = \sum_{m\in \ZZ}(q^{n-1}X)^m \int_{x\in F^j}\int_{a\in U_F}W\left[ 
\left(
\begin{smallmatrix}
\varpi^ma & 0 & 0 \\
x & I_j & 0  \\
0 & 0 & I_{n-j-1}
\end{smallmatrix}
\right)
\right] d^{\times}a dx ,$$
and $Z(W,X) = Z(W,X;0)$
\end{defn}

We first show that $Z(W,X;0)$ defines an element of $A[[X]][X^{-1}]$.
\begin{lemma}
\label{whitcptsupp}
Let $W$ be any element of $\Ind_{N_n}^G\psi$.  Then there exists an integer $N<0$ such that $W(\begin{smallmatrix}a&0\\0&I_{n-1}\end{smallmatrix})$ is zero for $v_F(a)<N$. Moreover if $W$ is compactly supported modulo $N_n$, then there exists an integer $L>0$ such that $W(\begin{smallmatrix}a&0\\0&I_{n-1}\end{smallmatrix})$ is zero for $v_F(a)>L$
\end{lemma}
\begin{proof}
There is an integer $j$ such that $\left(\begin{smallmatrix}1 & \mathfrak{p}^j & 0\\ 0 & 1 &0 \\ 0 & 0 & I_{n-2} \end{smallmatrix}\right)$ stabilizes $W$.  For $x$ in $\mathfrak{p}^j$, we have 
$$W\left(\begin{smallmatrix}a & 0 & 0\\ 0 & 1 &0 \\ 0 & 0 & I_{n-2} \end{smallmatrix}\right)=W\left(\left(\begin{smallmatrix} a & 0 &0 \\ 0 & 1 & 0 \\ 0 & 0 & I_{n-2} \end{smallmatrix}\right)\left(\begin{smallmatrix}1 &x & 0\\ 0 & 1 &0 \\ 0 & 0 & I_{n-2} \end{smallmatrix}\right)\right)
=\psi\left(\begin{smallmatrix}1 & ax & 0\\ 0 & 1 &0 \\ 0 & 0 & I_{n-2} \end{smallmatrix}\right)W\left(\begin{smallmatrix} a & 0 & 0\\ 0 & 1 & 0\\0&0&I_{n-2} \end{smallmatrix}\right)$$
Whenever $v_F(a)$ is negative enough that $ax$ lands outside of $\ker \psi=\mathfrak{p}$, we get that $\psi\left(\begin{smallmatrix}1 & ax & 0\\ 0 & 1 &0 \\ 0 & 0 & I_{n-2} \end{smallmatrix}\right)$ is a nontrivial $p$-power root of unity $\zeta$ in $W(k)$, hence $1-\zeta$ is the lift of something nonzero in the residue field $k$, and defines an element of $W(k)^{\times}$. This shows that $W(\begin{smallmatrix}a&0\\0&I_{n-1}\end{smallmatrix}) = 0$.  
\end{proof}

Just as in \cite{jps1}, the next two lemmas show that $Z(W,X;j)$ defines an element of $A[[X]][X^{-1}]$ when $0<j<n-2$ by reducing it to the case of $Z(W,X;0)$.

\begin{lemma}[\cite{jps1} Lemma 4.1.5]
Let $H$ be a function on $G$, locally fixed under right translation by $G$, and satisfying $H(ng) = \psi(n) H(g)$ for $g\in G$, $n\in N_n$.  Then the support of the function on $F^j$ given by
$$x\mapsto H\left[ \left(\begin{smallmatrix} a & 0 & 0 \\ x & I_j & 0  \\ 0 & 0 & I_{n-j-1}    \end{smallmatrix} \right) \right]$$ is contained in a compact set independent of $a\in F^{\times}$.
\end{lemma}
\begin{cor}
\label{j=0}
If $\rho$ denotes right translation $(\rho(g)\phi)(x)=\phi(xg)$, and $U$ is the unipotent radical of the standard parabolic subgroup of type $(1,n-1)$, then there is a finite set of elements $u_1,\dots,u_r$ of $U$ such that $$Z(W,X;j) = \sum_{i=1}^r Z(\rho( ^tu_i)W, X; 0)$$ for any $W\in \Ind_{N_n}^G\psi$.
\end{cor}

In \cite{jps1}, the zeta integrals form elements of the field $\CC((X))$ and it is proved that in fact they are elements of the subfield $\CC(X)$ of rational functions.  Whereas $\CC((X))$ (resp. $\CC(X)$) is the fraction field of the domain $\CC[[X]]$ (resp. $\CC[X,X^{-1}]$), our rings $A[[X]][X^{-1}]$ and $A[X,X^{-1}]$ are not in general domains.  The first main result of this paper is determining the appropriate fraction ring of $A[X,X^{-1}]$ in which the zeta integrals $Z(W,X;j)$ live:

\begin{thm}
\label{rationality}
Suppose $A$ is a Noetherian $W(k)$-algebra.  Let $S$ be the multiplicative subset of $A[X,X^{-1}]$ consisting of polynomials whose first and last coefficients are units.  Then if $V$ is admissible, Whittaker type, and $G$-finite, then $Z(W,X;j)$ lies in $S^{-1}A[X,X^{-1}]$ for all $W$ in $\cW(V,\psi)$ for $0\leq j \leq n-2$.
\end{thm}

In particular, the result holds if $V$ is co-Whittaker, as in Theorem \ref{introrationality}. The proof of Theorem \ref{rationality} will occupy the remainder of this section. The key idea is that the zeta integrals $Z(W,X)$ are completely determined by the values $W\left(\begin{smallmatrix}a & 0 \\0 & I_{n-1}\end{smallmatrix}\right)$ for $a\in F^{\times}$, and as $W$ ranges over $\cW(V,\psi)$, the set of these values is equivalent to the data of the $P_2$-representation $(\Phi^-)^{n-2}\cK$. Determining the rationality of $Z(W,X)$ will then reduce to a finiteness result for the quotient $\cK^{(n-1)}$, or more generally for $V^{(n-1)}$.

\subsection{Proof of Rationality}
\label{proofofrationality}

Denote by $\tau$ the right translation representation of $G_1$ on $\cK^{(n-1)}$. Let $B$ be the commutative $A$-subalgebra of $\End_A(\cK^{(n-1)})$ generated by $\tau(\varpi)$ and $\tau(\varpi^{-1})$, where $\varpi$ is a uniformizer of $F$. It follows from Corollary \ref{localn-1derivativesfinite} that $\cK_{\fp}^{(n-1)}$ is finitely generated  over $A_{\fp}$. For every $\fp$ of $\Spec A$, the inclusion $B_{\fp}\subset \End(\cK^{(n-1)})_{\fp}\hookrightarrow \End(\cK_{\fp}^{(n-1)})$, shows $B_{\fp}$ is finitely generated as an $A_{\fp}$-module.

\begin{lemma}
$B$ is finitely generated as an $A$-module.
\end{lemma}
\begin{proof}
$B$ is the image of the map $A[X,X^{-1}]\rightarrow \End_A(\cK^{(n-1)})$ sending $X$ to $\tau(\varpi)$. $B_{\fp}$ is the image of the localized map $A_{\fp}[X,X^{-1}]\rightarrow \big(\End_A(\cK^{(n-1)})\big)_{\fp}$, which is finitely generated. Thus for every $\fp$, $\tau(\varpi)$ and $\tau(\varpi^{-1})$ satisfy monic polynomials $s_{\fp}(X)$, $t_{\fp}(X)$ with coefficients in $A_{\fp}$. Since $s_{\fp}$ and $t_{\fp}$ have finitely many coefficients there exists a global section $f_{\fp}\notin \fp$ such that $s_{\fp}(X)$, $t_{\fp}(X)$ lie in $A_{f_{\fp}}[X]$. The open subsets $D(f_{\fp})$ cover $\Spec A$ and we can take a finite subset $\{f_1,\dots,f_n\}\subset \{f_{\fp}\}$ such that $(f_i)=1$. Since $\tau(\varpi)$ and $\tau(\varpi^{-1})$ satisfy monic polynomials over $A_{f_i}$, we have that $B_{f_i}$ is finitely generated over $A_{f_i}$ for each $i$. It follows that $B$ is finitely generated over $A$.
\end{proof}

Since $B$ is finitely generated over $A$, $\tau(\varpi)$ and $\tau(\varpi^{-1})$ satisfy monic polynomials $c_0+c_1X+\dots c_{r-1}X^{r-1}+X^r$ and $b_0 + b_1X + \dots + b_{s-1}X^{s-1}+X^s$ respectively. The degrees $r$ and $s$ are nonzero because $\tau(\varpi)$ and $\tau(\varpi^{-1})$ are units in $B$. Adding these together we have
$$0=\tau(\varpi)^{-s}+ b_{s-1}\tau(\varpi)^{-s+1}+\dots + b_0+ c_0+\dots c_{r-1}\tau(\varpi)^{r-1}+\tau(\varpi)^r,$$ hence $\tau(\varpi)$ satisfies a Laurent polynomial whose first and last coefficients are units.

The final ingredient in proving rationality is the following transformation property.
\begin{lemma}
\label{translationlemma}
$Z(\varpi^nW,X) = X^{-n}Z(W,X)$ for any $W \in \mathcal{W}(V,\psi)$, and any integer $n$.
\end{lemma}
\begin{proof}[Proof of Lemma]
Denote by $b_m$ the coefficient $\int_{U_F}W(\varpi^mu)d^{\times}u$. Then $Z(\varpi^nW,X)$ is $\sum_{m\in \ZZ}X^mb_{m+n}$, which can be rewritten $X^{-n}Z(W,X)$.
\end{proof}

\begin{proof}[Deducing Theorem \ref{rationality}]
The representation $\cK^{(n-1)}$ is formed by restricting the right translation representation on $(\Phi^-)^{n-2}\cK$ from $P_2$ to $G_1$, then taking the quotient by the $G_1$-stable submodule $(\Phi^-)^{n-2}\cK(U_2,\textbf{1})$. By Corollary \ref{explicitphi}, the right translation representation on $(\Phi^-)^{n-2}\cK$ is given by translations of the \emph{restricted} Kirillov functions $W|_{(\begin{smallmatrix}x&0\\0&I\end{smallmatrix})}$, denoted $W(x)$ for short. As an endomorphism of the quotient module $\cK^{(n-1)}$, $\tau(\varpi)$ satisfies a polynomial $X^n - a_{n-1}X^{n-1}-\cdots - a_1X -a_0$ (in fact we can take $a_0$ to be $-1$). Hence for any restricted Kirillov function $W(x)$ we have $$\varpi^nW(x) = \sum_{i=0}^{n-1}a_i\varpi^iW(x) + W_1(x),$$ for some element $W_1$ of $((\Phi^-)^{n-2}\cK)(U_2,\bf{1})$. Therefore we get a relation $$Z(\varpi^nW,X) = \sum_{i=0}^{n-1}a_iZ(\varpi^iW,X) + Z_1(X)$$ with $Z_1(X)$ being a Laurent polynomial by Lemma \ref{whitcptsupp}.  Using Lemma \ref{translationlemma}, then multiplying through by $X^n$ and rearranging we get $Z(W,X)(1-\sum_{i=0}^{n-1}a_iX^{n-i})=X^nZ_1(X)$ which completes the proof since $a_0$ is a unit.
\end{proof}

\section{Functional Equation and Gamma Factor}

\subsection{Contragredient Whittaker Functions}
\label{analogueoffouriertransform}
There is an analogue of the contragredient which is reflected on the level of Whittaker functions by a transform $\widetilde{(-)}$; the functional equation will relate the zeta integral of $W$ to that of its transform. We will need the following two matrices:
\begin{align*}
w=\left(\begin{smallmatrix} 
 0 & \cdots & 0 & 1\\
 0 & \cdots & -1 & 0\\
  & \vdots &  &  \\
  (-1)^{n-1}   &   \cdots &  0&0
\end{smallmatrix}\right),
& &
w'=\left(\begin{smallmatrix} 
(-1)^n & 0 & \cdots & 0\\
0 & 0 & \cdots & (-1)^{n-2}\\
\vdots   &     &   \vdots   &  \\
0 & (-1)^0 & \cdots & 0
\end{smallmatrix}\right)
\end{align*}
For any element $W$ of $\Ind_{N_n}^G\psi$, define the transform $\widetilde{W}$ of $W$ as $\widetilde{W}(g):=W(w g^{\iota})$, where $g^{\iota}:= ^tg^{-1}$.

\begin{obs}
\label{tildestaysinind}
If $V$ is of Whittaker type, then for $v\in V$, $\widetilde{W_v}$ is an element of $\Ind_N^G\psi$ because 
$$\widetilde{W}(ng) = W(w(ng)^{\iota}) = W(wn^{\iota}w^{-1}wg^{\iota})=\psi(wn^{\iota}w^{-1})W(wg^{\iota}) = \psi(n)\widetilde{W}(g).$$
\end{obs}

Thus $Z(\widetilde{w'W},X;j)$ lands in $A[[X]][X^{-1}]$ by Lemma \ref{whitcptsupp}. In this section we state the second main result and recover the rationality properties of Section \ref{whittakerandkirillovfunctions} for $Z(\widetilde{w'W},X;j)$. The second main result is as follows:

\begin{thm}
\label{rikka}
Suppose $A$ is a Noetherian $W(k)$-algebra, and suppose $V$ in $\Rep_A(G)$ is co-Whittaker and primitive. Let $S$ denote the multiplicative subset of Theorem \ref{rationality}. Then there exists a unique element $\gamma(V,X,\psi)$ of $S^{-1}A[X,X^{-1}]$ such that for any $W\in \cW(V,\psi)$,
$$Z(W,X; j)\gamma(V,X,\psi) = Z(\widetilde{w'W},\frac{1}{q^nX};n-2-j)$$ for $0\leq j\leq n-2$.
\end{thm}

The proof of Theorem \ref{rikka} is in Section \ref{universalgammafactors}. We now verify that $Z(\widetilde{w'W},\frac{1}{q^nX};j)$ always lives in $S^{-1}A[X,X^{-1}]$.

\begin{prop}
\label{viotanice}
Suppose $V$ in $\Rep_A(G)$ is admissible, Whittaker type, $G$-finite, Schur, and primitive.  Let $V^{\iota}$ denote the smooth $A[G]$-module whose underlying $A$-module is $V$ and whose $G$-action is given by $g\cdot v = g^{\iota}v$.  Then $V^{\iota}$ is also admissible, Whittaker type, $G$-finite, Schur, and primitive.
\end{prop}
\begin{proof}
Consider the map $\Hom_{N_n}(V,\psi)\rightarrow \Hom_{N_n}(V^{\iota},\psi)$ given by $\lambda\mapsto \widetilde{\lambda}$, where $\widetilde{\lambda}:x\mapsto \lambda(wx)$.
We have $\widetilde{\lambda}(n\cdot v) = \lambda(wn^{\iota}w^{-1}wv) = \psi(n)\widetilde{\lambda}(v)$, which shows $\widetilde{\lambda}$ indeed defines an element of $\Hom_{N_n}(V^{\iota},\psi)$.  Since $w^2 = (-1)^{n-1}I_n$, it is an isomorphism of $A$-modules. Admissibility, $G$-finiteness, and Schur-ness all hold for $V^{\iota}$ since $g\mapsto g^{\iota}$ is a topological automorphism of the group $G$. Since $V$ is Schur, $A$ must be connected, hence $V$ must be primitive since it is Schur.
\end{proof}

In particular, $(V^{\iota})^{(n)} = V^{\iota}/V^{\iota}(N_n,\psi)$ is free of rank one and we may define $(\widetilde{W})_v(g) = \widetilde{\lambda}(g^{\iota}v)$ and take $\cW(V^{\iota},\psi)$ to be the $A$-module $\{(\widetilde{W})_v:v\in V^{\iota}\}$ as before.  Note that this is precisely the same as $\{\widetilde{(W_v)}:v\in V\}$.  We record this simple observation as a Lemma:

\begin{lemma}
\label{iotawhittaker}
If $\lambda$ is a generator of $\Hom_{N_n}(V,\psi)$ then  $\widetilde{\lambda}:x\mapsto \lambda(wx)$ is a generator of $\Hom_{N_n}(V^{\iota},\psi)$ and defines $\cW(V^{\iota},\psi)$.  There is an isomorphism of $G$-modules
$\cW(V,\psi)\rightarrow \cW(V^{\iota},\psi)$ sending $W$ to $\widetilde{W}$.
\end{lemma}

Thus all the hypotheses for the results of the previous sections, in particular Theorem \ref{rationality}, apply to $V^{\iota}$ whenever they apply to $V$, so we get $Z(\widetilde{w'W},X;j)$ is in $S^{-1}A[X,X^{-1}]$.  Now we can make the substition $\frac{1}{q^nX}$ for $X$ in the ratio of polynomials $Z(\widetilde{w'W},X;j)$ to get $Z(\widetilde{w'W},\frac{1}{q^nX};j)$.  It will again be in $S^{-1}A[X,X^{-1}]$ because this process swaps the first and last coefficients in the denominator (and $q$ is a unit in $A$ since $q$ is relatively prime to $\ell$).

\subsection{Zeta Integrals and Tensor Products}
\label{zetaintegralsbasechange}

The goal of this subsection is to check that the formation of zeta integrals commutes with change of base ring $A$. For any $f:A\rightarrow B$, denote by $\psi_A\otimes B$ the free rank one $B$-module with action given by the character $f\circ \psi$.  The group action on $V\otimes_A B$ is given by acting in the first factor.  Let $i$ denote the map $V\rightarrow V\otimes_A B$.  Proposition \ref{functorprop} (6), gives the following lemma.

\begin{lemma} 
\begin{enumerate}[(1)]
\item If $V$ is of Whittaker type, so is $V\otimes_A B$.
\item Let $\lambda$ generate $\Hom_{A[N]}(V,\psi)$ as an $A$-module.  Then $\lambda\otimes id$ is a generator of $\Hom_{B[N]}(V\otimes B,\psi\otimes B)$.
\item Let $W_{v\otimes b}(g) := (f \circ \lambda)(gv)\otimes b$ define elements of $\cW(V\otimes B,\psi\otimes B)$.  Then $f\circ W_v = W_{i(v)}$ for any $v\in V$.
\end{enumerate}
\end{lemma}

From the definition of integration given in \S \ref{notationconventions}, it follows that if $\Phi_k$ is the characteristic function of some $H_k$, then $\int (f\circ\Phi_k) d(f\circ\mu^{\times})= (f\circ \mu^{\times})(H_k) = f \left(\int \Phi_k d(\mu^{\times})\right)$.  It  follows from the definitions that $(f\circ \widetilde{W})(x) = \widetilde{f\circ W}(x)$.  

\begin{cor}
\label{zetacommutestensor}
Let $F$ denote the map of formal Laurent series rings $A[[X]][X^{-1}]\rightarrow B[[X]][X^{-1}]$ induced by $f$, then we have
\begin{align}
F\left(Z(W_v,X;j) \right) &= Z(f\circ W, X;j) =Z(W_{i(v)}, X;j)\\
F\left(Z(\widetilde{w'W},X; j) \right) &= Z(f\circ \widetilde{w'W}, X;j) = Z(\widetilde{w'(f\circ W)},X;j)
\end{align}
for any $W$ in $\cW(V,\psi)$, and for $0\leq j \leq n-2$.
\end{cor}

The next proposition follows from the linearity of the zeta integrals and the transform $\widetilde{(-)}$.
\begin{prop}
\label{imageofi}
Suppose there is an element  $\gamma(V, X, \psi)$ in $A[[X]][X^{-1}]$ satisfying a functional equation as in Theorem \ref{rikka} for all $W_v\in \cW(V,\psi)$.  Then the element $F(\gamma(V, X, \psi))\in B[[X]][X^{-1}]$ satisfies the functional equation for all $W\in \cW(V\otimes B, \psi\otimes B)$.
\end{prop}

\subsection{Construction of the Gamma Factor}
We define the gamma factor to be what it must in order to satisfy the functional equation of Theorem \ref{rikka} for a single, particularly simple Whittaker function $W_0$. We seek a $W_0$ such that $Z(W,X;0)$ is a unit in $S^{-1}A[X,X^{-1}]$.

By Proposition \ref{schwarzandkirillov} and Lemma \ref{phicommuteswithK}, we have that $\cInd_{U_2}^{P_2}\psi \subset (\Phi^-)^{n-2}\cK$.  Since $\cInd_{U_2}^{P_2}\psi$ is isomorphic to $C_c^{\infty}(F^{\times},A)$ via restriction to $G_1$ (recall that $C_c^{\infty}(F^{\times},A)$ denotes the locally constant compactly supported functions $F^{\times}\rightarrow A$), we find the following:

\begin{prop}
\label{charfn}
Suppose $V$ in $\Rep_A(G)$ is of Whittaker type.  Then the characteristic function of $U_F^1$ is realized as a restricted Whittaker function $W_0(\begin{smallmatrix}g & 0 \\ 0 & I_{n-1}\end{smallmatrix})$ for some $W_0$ in $\cW(V,\psi)$.
\end{prop}

From now on, the symbol $W_0$ will denote a choice of element in $\cW(V,\psi)$ whose restriction to $ (\begin{smallmatrix}g & 0 \\ 0 & I_{n-1}\end{smallmatrix})$ is the characteristic function of $U_F^1$.  Then $Z(W_0,X)$ is $\int_{U_F}W_1(\begin{smallmatrix} a & 0 \\ 0& 1\end{smallmatrix})d^{\times}a = \mu^{\times}(U_F^1) = 1$.  Since we want our gamma factor to satisfy the functional equation for $W_0$, we are left with no choice:

\begin{defn}[The Gamma Factor]
\label{gammafactor}
Let $A$ be any Noetherian $W(k)$-algebra and suppose $V$ in $\Rep_A(G)$ is of Whittaker type.  We define the gamma factor of $V$ with respect to $\psi$ to be the element of $A[[X]][X^{-1}]$ given by $\gamma(V,X,\psi):= Z(\widetilde{w'W_0},\frac{1}{q^nX};n-2)$.
\end{defn}

When $V$ is co-Whittaker and primitive the uniqueness of this gamma factor will follow from the functional equation: if $\gamma$ and $\gamma'$ both satisfy the functional equation for all Whittaker functions, then $\gamma = Z(\widetilde{w'W_0},\frac{1}{q^nX},n-2) = \gamma'$. In particular for such representations our construction of the gamma factor does not depend on the choice of $W_0$.

\subsection{Functional Equation for Characteristic Zero Points}
If the residue field $\kappa(\fp)$ of $\fp$ has characteristic zero, the reduction modulo $\fp$ of $Z(W,X;j)$ forms an element of $\overline{\kappa(\fp)}(X)$. As $\overline{\kappa(\fp)}$ is an uncountable algebraically closed field of characteristic zero, we may fix an embedding $\CC\hookrightarrow \overline{\kappa(\fp)}$. The proof of \cite[Thm 2.7(iii)(2)]{jps2} (which occurs in \cite[\S 2.11]{jps2}) carries over verbatim to the setting where $\pi$ and $\pi'$ are admissible, Whittaker type, $G$-finite representations over any field containing $\CC$, hence for representations over $\overline{\kappa(\fp)}$. Thus the reduction modulo $\fp$ of $\Psi(W,X;j)$ is precisely the integral $\Psi(s,W;j)$ of \cite[\S 4.1]{jps1}, after replacing the complex variable $q^{-(s+\frac{n-1}{2})}$ with the indeterminate $X$, and there exists a unique element, which we will call $\gamma_{\fp}(s,V\otimes \overline{\kappa(\fp)},\psi)$, in $\overline{\kappa(\fp)}(q^{-s})$ such that \emph{for all} $W\in \mathcal{W}(V_{\fp}\otimes\overline{\kappa(\fp)},\psi_{\fp})$ and for all $j\geq 0$, $$\Psi(1-s,\widetilde{w'W};n-2-j) = \gamma_{\fp}(s,V\otimes\overline{\kappa(\fp)},\psi_{\fp})\Psi(s,W,j).$$  The change of variable $s\mapsto 1-s$ can be re-written as $-(s+\frac{n-1}{2})\mapsto (s+\frac{n-1}{2})-n$, so writing the functional equation in terms of $X$ we have shown the following Lemma:
\begin{lemma}
\label{specialgammas}
Suppose $V$ is admissible of Whittaker type, and $G$-finite. For each prime $\fp$ of $A$ with residue characteristic zero, there exists a unique element $\gamma_{\fp}(V\otimes \kappa(\fp),X,\psi_{\fp})$ in $\kappa(\fp)(X)$ such that for all $W$ in $\mathcal{W}(V\otimes\kappa(\fp),\psi_{\fp})$ and for $0\leq j \leq n-2$ we have $$Z(\widetilde{w'W},\frac{1}{q^nX};n-2-j) = \gamma_{\fp}(X,V\otimes \kappa(\fp),\psi_{\fp})Z(W,X;j).$$  Moreover, $\gamma_{\fp}(V\otimes\kappa(\fp),X,\psi_{\fp}) = \gamma(V,X,\psi) \mod \fp$ by uniqueness in \cite{jps1}.
\end{lemma}

\subsection{Proof of Functional Equation When $A$ is Reduced and $\ell$-torsion Free}
\label{proofoffunctionalequationwhenAisreduced}
In the case that $A$ is reduced and $\ell$-torsion free as a $W(k)$-algebra, we get a slightly stronger result than that of Theorem \ref{rikka}.
\begin{thm}
\label{rikkaforreducedflat}
If $A$ is a Noetherian $W(k)$-algebra and $A$ is reduced and $\ell$-torsion free, then the conclusion of Theorem \ref{rikka} holds for any $V$ in $\Rep_A(G)$ which is $G$-finite, and admissible of Whittaker type.
\end{thm}
\begin{proof}
Let $\fp$ be any characteristic zero prime, and let $f_{\fp}:A\rightarrow \kappa(\fp)$ be reduction modulo $\fp$. Corollary \ref{zetacommutestensor} and Lemma \ref{specialgammas} tell us that $$f_{\fp}\left(\gamma(V,X,\psi)Z(W,X) - Z(\widetilde{w'W}, \frac{1}{q^nX};n-2)\right)=0$$ for \emph{any} $W$ in $\cW(V,\psi)$, not just $W_0$.  This shows that the difference $$\gamma(V,X,\psi)Z(W,X) - Z(\widetilde{w'W},\frac{1}{q^nX};n-2)$$ is in the intersection of all characteristic zero primes of $A$. When $A$ is reduced its zero divisors are the union of its minimal primes, so it is $\ell$-torsion free if and only if all minimal primes have residue characteristic zero. Thus when $A$ is reduced and $\ell$-torsion free, the intersection of all characteristic zero primes of $A$ equals zero, so the functional equation holds for any $W$ in $\cW(V,\psi)$.

We now prove uniqueness. If there were another element $\gamma'$ satisfying the same property, it would satisfy the functional equation in $\kappa(\fp)$ for all $W_{i(v)}$ by reduction, so it satisfies the functional equation for all $W$ in $\cW(V\otimes\overline{\kappa(\fp)},\psi_{\fp})$. But uniqueness in Lemma \ref{specialgammas} then shows $f_{\fp}(\gamma(V,X,\psi) - \gamma') = 0$ for all characteristic zero primes $\fp$ of $A$. Again, this means $\gamma' = \gamma(V,X,\psi)$.

We get rationality by observing that whenever $V$ is admissible of Whittaker type, it has a canonical co-Whittaker submodule $T$ by Proposition \ref{cowhittakersub}, which is primitive if $V$ is primitive. Since $\gamma(T,X,\psi)$ satisfies the functional equation for all $W$ in $\cW(T,\psi)$, we must have $\gamma(T,X,\psi)=\gamma(V,X,\psi)$ by the construction of the gamma factor. But $\gamma(T,X,\psi)$ is in $S^{-1}A[X,X^{-1}]$ by Theorem \ref{rationality}, which holds for primitive co-Whittaker modules.
\end{proof}

\section{Universal Gamma Factors}
\label{universalgammafactors}
When $V$ is primitive and co-Whittaker, we can remove the hypothesis that $A$ is reduced and $\ell$-torsion free by specializing the gamma factor for the universal co-Whittaker module $e\fW$. 
\begin{thm}[\cite{h_bern} Thm 12.1]
Any block $e\cZ$ of the Bernstein center of $\Rep_{W(k)}(G)$ is a finitely generated (hence Noetherian), reduced, $\ell$-torsion free $W(k)$-algebra.
\end{thm}
By Proposition \ref{dominance}, $e\fW$ is co-Whittaker, and since it is clearly primitive, all the hypotheses of Proposition \ref{rikkaforreducedflat} are satisfied. Hence (Thm \ref{rikkaforreducedflat}) there exists a unique gamma factor in $S^{-1}(e\cZ[X,X^{-1}])$, which we will denote $\Gamma(e\fW,X,\psi)$, satisfying the functional equation for all $W$ in $\cW(e\fW,\psi)$.

\begin{proof}[Proof of Theorem \ref{rikka}] Since $V$ is primitive and co-Whittaker, there is a (unique) primitive idempotent $e$ of $\cZ$ and a ring homomorphism $f_V:e\cZ\rightarrow \End_G(V)\isomto A$, and a surjection of $A[G]$-modules $e\fW\otimes_{f_V}A\rightarrow V$ preserving the top derivative, so that $f_V(\Gamma(e\fW,X,\psi)) =\gamma(e\fW\otimes_{f_V}A,X,\psi)$. Since $\Gamma(e\fW,X,\psi)$ satisfies the functional equation for all $W$ in $\cW(e\fW,\psi)$, we can apply Proposition \ref{imageofi} again to conclude that $\gamma(e\fW\otimes A,X,\psi)$ satisfies the functional equation for all $W$ in $\cW(e\fW\otimes A,\ \psi)$. Since $e\fW\otimes A$ has a surjection onto $V$ preserving the top derivative, Lemma \ref{whittakerunderhomomorphism} tells us that $\cW(V,\psi)= \cW(e\fW\otimes A, \psi)$. The functional equation shows that  Definition \ref{gammafactor} gives a unique gamma factor, hence $\gamma(V,X,\psi) = \gamma(e\fW\otimes A,X,\psi)$; it satisfies the functional equation for all $W$ in $\cW(V,\psi)$. Note that since $\Gamma(e\fW,X,\psi)$ is in $S^{-1}(e\cZ[X,X^{-1}])$, its image in $f_V$  is in the corresponding fraction ring of $A[X,X^{-1}]$. This proves Theorem \ref{rikka}.
\end{proof}

We can extend the uniqueness and rationality result to a larger class of representations, though with a weaker functional equation coming only from the co-Whittaker case:

\begin{cor}
\label{mainthmbroader}
Let $V$ be admissible, primitive, of Whittaker type and let $T$ be its canonical co-Whittaker submodule. Then there exists a unique gamma factor $\gamma(V,X,\psi)$ in $S^{-1}(A[X,X^{-1}])$ which equals $\gamma(T,X,\psi)$, and satisfies the functional equation for all $W$ in $\cW(T,\psi)$.
\end{cor}
\begin{proof}
When $V$ is admissible of Whittaker type it has a canonical co-Whittaker sub by Proposition \ref{cowhittakersub}. We have just shown that its gamma factor $\gamma(T,X,\psi)$ satisfies the functional equation for all $W$ in $\cW(T,\psi)$. Applying Proposition \ref{whittakerunderhomomorphism} with $\alpha:T\rightarrow V$ being the inclusion map, we conclude that $\cW(T,\psi) \subset \cW(V,\psi)$. The coefficients of the series $Z(\widetilde{w'W_0},\frac{1}{q^nX};n-2)$ in Definition \ref{gammafactor} are determined by $G$-translates of the Whittaker function $W_0$, so occurs already in $\cW(T,\psi)$, so by definition $\gamma(T,X,\psi) = \gamma(V,X,\psi)$. In particular $\gamma(V,X,\psi)$ lies in $S^{-1}A[X,X^{-1}]$ and satisfies the functional equation for all $W$ in $\cW(T,\psi)$.
\end{proof}

We can make precise the sense in which we have created a universal gamma factor:

\begin{cor}
\label{universalgammatheorem}
Suppose $A$ is a Noetherian $W(k)$-algebra, and suppose $V$ is a co-Whittaker $A[G]$-module in the subcategory $e\Rep_{W(k)}(G)$ of $\Rep_{W(k)}(G)$. Then there is a homomorphism $f_V:e\cZ\rightarrow A$ and $f_V(\Gamma(e\fW,X,\psi))$ equals the unique $\gamma(V,X,\psi)$ satisfying a functional equation for all $W$ in $\cW(V,\psi)$.
\end{cor}

Again, we can broaden the class of representations at the cost of a more restrictive functional equation:

\begin{thm}
\label{universalgammatheorembroader}
Suppose $A$ is any Noetherian $W(k)$-algebra, and suppose $V$ is an admissible $A[G]$-module of Whittaker type in the subcategory $e\Rep_{W(k)}(G)$. Then there is a homomorphism $f_V:e\cZ\rightarrow A$ and the gamma factor of Corollary \ref{mainthmbroader} equals $f_V(\Gamma(e\fW,X,\psi))$.
\end{thm}
\begin{proof}
We define $f_V$ to be the homomorphism $e\cZ\rightarrow \End_G(T) \isomto A$ where $T$ is the canonical co-Whittaker submodule of Proposition \ref{cowhittakersub}. Since $T$ lies in $e\Rep_{W(k)}(G)$, $e\fW\otimes_{f_V}A$ surjects onto $T$, and we have $f_V(\Gamma(e\fW,X,\psi) = \gamma(T,X,\psi)$ (Prop \ref{whittakerunderhomomorphism}), and since $T$ injects into $V$ (with top derivative preserved), again by Prop \ref{whittakerunderhomomorphism}, we have
$$f_V(\Gamma(e\fW,X,\psi)) = \gamma(e\fW\otimes_{e\cZ,f_V}A,X,\psi) = \gamma(T,X,\psi) = \gamma(V,X,\psi).$$
\end{proof}

\bibliography{mybibliography}{}
\bibliographystyle{alpha}
\end{document}